\documentclass[12pt]{amsart}
\usepackage{latexsym}
 \usepackage{slashed}

\usepackage{amssymb,amsmath,amsthm,amscd, amssymb,
graphicx, enumerate, verbatim, mathrsfs, color, url}
\usepackage{hyperref}
\usepackage[all]{xy}

\numberwithin{equation}{section}
\newtheorem{cor}[equation]{Corollary}
\newtheorem{lem}[equation]{Lemma}

\newtheorem{thm}[equation]{Theorem}

\newtheorem{Example}[equation]{Example}

\newtheorem{remark}[equation]{Remark}
\newenvironment{rmk}{\begin{remark}\rm}{\end{remark}}
\def\co{\colon\thinspace}

\newcommand{\const}{\mbox{const.}}

\newcommand{\dvol}{\mbox{dvol}}

\newcommand{\vol}{\operatorname{vol}}

\newcommand{\diam}{\operatorname{diam}}
\newcommand{\Ric}{\mbox{Ric}}
\newcommand{\Rm}{\mbox{Rm}}
\newcommand{\Iso}{\operatorname{Iso}}

\def\g{\gamma}

\def\eps{\varepsilon}

\def\Z{\mathbb{Z}}

\newcommand{\R}{{\mathbb R}}

\newcommand{\C}{{\mathbb C}}

\makeatletter
\def\equalsfill{$\m@th\mathord=\mkern-7mu
\cleaders\hbox{$\!\mathord=\!$}\hfill
\mkern-7mu\mathord=$}
\makeatother

\begin{document}

\abovedisplayskip=6pt plus3pt minus3pt
\belowdisplayskip=6pt plus3pt minus3pt

\title[On noncollapsed almost Ricci-flat 4-manifolds]
{\bf On noncollapsed almost Ricci-flat 4-manifolds}

%\thanks{\it 2000 Mathematics Subject classification.\rm\ 
%Primary 53C20, 

%\it\ Keywords:\rm\ nonnegative curvature, soul, 
%pseudoisotopy, space of metrics, diffeomorphism group.}\rm

\author{Vitali Kapovitch \and John Lott}

\address{Vitali Kapovitch\\Department of Mathematics\\
University of Toronto\\
Toronto, ON, Canada M5S 2E4\\}\email{vtk@math.toronto.edu}

\address{John Lott\\Department of Mathematics\\
University of California-Berkeley\\
Berkeley, CA  94720-3804\\}\email{lott@berkeley.edu}

\thanks{ 
The first author was partially supported by an NSERC Discovery Grant. 
The second author was partially supported by NSF grant  
DMS-1510192. 
Research at MSRI was partially supported by 
NSF grant DMS-1440140. We thank MSRI for its hospitality during the 
Spring 2016 program.}
\date{October 13, 2017}
\begin{abstract}
We give topological conditions to ensure that a
noncollapsed almost Ricci-flat $4$-manifold admits a Ricci-flat
metric. One sufficient condition is that the manifold is
spin and has a nonzero $\hat A$-genus. Another condition
is that the fundamental group is infinite or, more generally, of
sufficiently large cardinality.
\end{abstract}
\maketitle
%\tableofcontents

\section{Introduction}

One of the most basic pinching theorems in Riemannian geometry says
that a noncollapsed almost flat manifold admits a flat metric. 
The word noncollapsed refers to a lower volume bound.
More precisely, given $n \in \Z^+$ and $v > 0$, there is some $\epsilon =
\epsilon(n,v) > 0$ so that if $(M,g)$ is a Riemannian $n$-manifold 
with $\vol(M) \ge v$, $\diam(M) \le 1$ and $|\Rm_M| \le \epsilon$,
then $M$ admits a flat Riemannian metric. Here $\Rm_M$ denotes the
Riemann curvature tensor. This result seems to have first been stated
by Gromov in \cite{Gr78}, where he noted that it follows from
Cheeger's arguments in \cite{Ch69}. (The point of \cite{Gr78} was to
characterize what happens when one removes the volume assumption.)

One can ask if there is an analogous statement for noncollapsed
almost Ricci-flat manifolds.  In dimension less than four, being
almost Ricci-flat is the same as being almost flat. Hence the first
interesting case is in dimension four. We give topological conditions
to ensure that a noncollapsed almost Ricci-flat $4$-manifold admits
a Ricci-flat metric. 
We also give more general results about noncollapsed manifolds with 
almost nonnegative Ricci curvature, or almost 
nonnegative scalar curvature and bounded Ricci curvature.

In the rest of the introduction we state the main results,
outline the proof of Theorem~\ref{main-thm}, mention some earlier
related results and give the structure of the paper.

\subsection{Statement of results}

The first result is in four dimensions.
Recall that
the $\widehat{A}$-genus $\widehat{A}(M)$ of a closed oriented manifold $M$
is a certain rational combination of the Pontryagin numbers of $M$.
In four dimensions, $\widehat{A}(M)$ equals minus one
eighth of the signature of $M$.
We consider a noncollapsed spin $4$-manifold with
almost nonnegative scalar curvature and a nonzero $\widehat{A}$-genus.
With an upper bound on the Ricci curvature, the conclusion is that
the manifold
must be diffeomorphic to a $K3$ surface.

Let $S$ denote scalar curvature.

\begin{thm}\label{main-thm}
Given $v>0$ and $\Upsilon < \infty$, 
there is an $\eps = \eps(v,\Upsilon) >0$ with the following property.
Suppose that $(M,g)$ is a closed connected Riemannian spin $4$-manifold with
$\hat A(M)\ne 0$, $\vol(M)\ge v$, $\diam(M)\le 1$, $\Ric_M \le \Upsilon g$ and
$S_M \ge - \epsilon$.
Then $M$ is diffeomorphic to a $K3$ surface. 
\end{thm}

As a consequence, if a noncollapsed almost Ricci-flat $4$-manifold
is spin, and has a nonzero $\widehat{A}$-genus, then it
admits a Ricci-flat metric.

\begin{cor} \label{main-cor}
Given $v>0$, there is an $\eps^\prime = \eps^\prime(v) >0$ 
with the following property.
Suppose that $(M,g)$ is a closed connected Riemannian spin $4$-manifold with
$\hat A(M)\ne 0$, $\vol(M)\ge v$, $\diam(M)\le 1$
and $|\Ric_M|\le \eps^\prime$.
Then $M$ is diffeomorphic to a $K3$ surface.  In particular, $M$ admits
a Ricci-flat metric.
\end{cor}

For example, if $M = K3 \# (S^2\times S^2)$
then there is no Riemannian
metric on $M$ with  $\vol(M)\ge v$, $\diam(M)\le 1$ and 
$|\Ric_M|\le \eps^\prime$.

If $\pi_1(M)$ is infinite, or of sufficiently large cardinality, 
then we have the following 
related 
$n$-dimensional result, which does not involve any spin assumption on $M$.

\begin{thm}\label{main-large-pi1}
Given $n \in \Z^+$,  $v>0$ and $\Lambda, \Upsilon < \infty$, there exist 
$\eps = \eps(n, v, \Lambda, \Upsilon)>0$, 
$C =C(n, v, \Lambda, \Upsilon) < \infty$,
$\eps^\prime = \eps^\prime(n, v, \Lambda)>0$ and 
$C^\prime =C^\prime(n, v, \Lambda) < \infty$ with the
following property. 
Suppose that
$(M,g)$ is an $n$-dimensional closed connected Riemannian manifold with
$\int_M | \Rm_M |^{\frac{n}{2}} \: \dvol_M \le \Lambda$,
$\vol(M)\ge v$ and $\diam(M)\le 1$.
\begin{enumerate}
\item If $- \epsilon g \le \Ric_M \le \Upsilon g$ and
$|\pi_1(M)|\ge C$ then $M$ admits a $W^{2,p}$-regular Riemannian
metric $h$ with nonnegative measurable Ricci curvature 
for which the universal cover $(\tilde{M}, \tilde{h})$
isometrically splits off an $\R$-factor.
{
In particular, $\pi_1(M)$ is infinite.
}
\item
If $ |\Ric_M|\le \eps^\prime$ and $|\pi_1(M)|\ge C^\prime$ then $M$ admits a 
Ricci-flat metric $h^\prime$ for which the universal cover 
$(\tilde{M}, \tilde{h}^\prime)$
isometrically splits off an $\R$-factor.
\end{enumerate}
\end{thm}

In the four dimensional case, the conclusions of 
Theorem \ref{main-large-pi1} can be
made more precise. Theorem \ref{main-large-pi1}(1) becomes a statement about manifolds with
almost
nonnegative Ricci curvature, again under an upper Ricci curvature
bound. 
Theorem \ref{main-large-pi1}(2) becomes
a second sufficient topological condition,
involving the fundamental group, for
a noncollapsed almost Ricci-flat $4$-manifold to admit a Ricci-flat metric.

\begin{cor} \label{cor-large-pi1}
Given  $v>0$ and $\Upsilon < \infty$, there exist $\eps^{\prime \prime} = 
\eps^{\prime \prime}(v, \Upsilon)>0$, 
$C^{\prime \prime} =C^{\prime \prime}(v, \Upsilon) < 
\infty$,  $\eps^{\prime \prime \prime} = 
\eps^{\prime \prime \prime}(v)>0$ and 
$C^{\prime \prime \prime} =C^{\prime \prime \prime}(v) < 
\infty$ with the
following property. Suppose that
$(M,g)$ is a closed connected Riemannian $4$-manifold with
$ \vol(M)\ge v$ and $\diam (M)\le 1$.
\begin{enumerate}
\item If $- \eps^{\prime \prime} g \le \Ric_M \le \Upsilon g$ and 
$|\pi_1(M)|\ge C^{\prime \prime}$ then $M$ admits a smooth
Riemannian metric with nonnegative Ricci curvature for which
the universal cover $(\tilde M, \tilde g)$ isometrically splits off
an $\R$-factor.
\item If  $|\Ric_M|\le \eps^{\prime \prime \prime}$ and 
$|\pi_1(M)|\ge C^{\prime \prime \prime}$ then $M$ admits a flat metric. 
\end{enumerate}
\end{cor}

As an example of Corollary \ref{cor-large-pi1}(2), 
let $M$ be the result of performing 
surgery along an embedded circle in $T^4$, i.e.
removing a tubular neighborhood of the circle
and attaching a copy of $D^2 \times S^2$.
Then there is no Riemannian
metric on $M$ with  $\vol(M)\ge v$, $\diam(M)\le 1$ and $|\Ric_M|\le 
\eps^{\prime \prime \prime}$.
(On the other hand, Anderson showed that if the circle is a meridian
curve in $T^4$, and one performs 
surgery with respect to the canonical trivialization of its normal bundle,
then the resulting manifold does admit a sequence of collapsing almost 
Ricci-flat metrics ~\cite[Theorem 0.4]{An92}.)

\begin{rmk}
The diffeomorphism types of the manifolds $M$ in the conclusion of
Corollary \ref{cor-large-pi1}(1) are easy to describe, using
\cite[Section 9]{Ha86}.
\end{rmk}

\begin{rmk}
The known Ricci-flat closed $4$-manifolds are flat (hence with infinite
fundamental group) or are finitely covered by a $K3$ surface
(which is spin and has $\widehat{A}(K3) = 2$). In view of this fact, the
topological hypotheses of Corollary \ref{main-cor} and 
Corollary \ref{cor-large-pi1}(2)
are not unreasonable.
\end{rmk}

Finally, we give a fundamental group restriction on noncollapsed manifolds
with almost nonnegative Ricci curvature.

\begin{thm}\label{thm-lower-ricci-noncol}
Given $n \in \Z^+$ and $v > 0$, there are 
$\eps = \eps(n,v) >0$ and $I = I(n,v) < \infty$
with the following property.
Suppose that $(M,g)$ is a closed connected Riemannian $n$-manifold with
$\vol(M) \ge v$, $\diam(M)\le 1$ and $\Ric_M \ge - \eps$.
Then $\pi_1(M)$ has an abelian subgroup (of 
index at most $I$) generated by at most $n$ elements.
\end{thm}

\subsection{Outline of the proof of Theorem~ \ref{main-thm}}

If the theorem fails then we 
take a sequence of counterexamples, with $\epsilon \rightarrow 0$.
The hypotheses imply that there is
a uniform lower Ricci curvature bound.
From a result of Cheeger and Naber, there is an {\em a priori} upper bound on
$\int_M |\Rm_M|^2 \: \dvol_M$ \cite{Ch-Na-codim4}. 
The first part of the
argument for Theorem ~\ref{main-thm} is now standard, based on ideas of
Anderson \cite{An90}, Anderson-Cheeger \cite{AC91} and Bando
\cite{Bando90}.
We can pass to a subsequence that converges, in the Gromov-Hausdorff topology,
to a four dimensional
orbifold $X$ that has a finite number of orbifold singular points, 
and a $C^{1,\alpha}$-regular Riemannian metric $g_X$ on its
regular part $X_{reg}$.
Doing appropriate blowups to zoom in on the formation of singular
points of $X$, one obtains noncompact Ricci-flat ALE orbifolds.  Doing
further blowups to zoom in on the formation of their singular points, 
one obtains a bubble tree whose vertices correspond to 
Ricci-flat orbifolds.  By
assembling the geometric pieces in the bubble tree, one can reconstruct
the diffeomorphism type of the original manifold.

In order to proceed, we need more information about
$X$ and the ALE blowups. 
The assumption that $M$ is spin, with nonzero $\widehat{A}$-genus, helps
in several ways. First,
the nonvanishing of the
$\widehat{A}$-genus of $M$ implies that $M$ has a nonvanishing harmonic spinor
field. We show that these harmonic spinor fields 
pass to a nonzero parallel spinor
field on $X_{reg}$. This forces $X_{reg}$ to be
smooth and Ricci-flat.
We show that $X$ is a Ricci-flat spin orbifold.
The existence of the parallel spinor field now implies that $X$ is 
a hyperK\"ahler orbifold.

The spin assumption on $M$ also implies that the Ricci-flat ALE blowup
orbifolds are spin.
Compatibility of spin structures, along with the existence of
the nonzero parallel spinor field on $X$, means that the 
blowup orbifolds have spinor fields that are
asymptotically parallel at infinity, and nonzero there.
Then a small variation on a result of
Nakajima \cite{Na90} says that the blowup orbifolds are
hyperK\"ahler.  One knows enough about hyperK\"ahler
ALE $4$-manifolds, and $4$-orbifolds, to conclude that after assembling
the pieces in the bubble tree, the result is diffeomorphic to a
compact hyperK\"ahler manifold.  Theorem ~\ref{main-thm} follows
from this fact.

\subsection{Related results}

Anderson gave a Ricci pinching result for Riemannian metrics on
$S^4$ and $\C P^2$ \cite[Theorem 1.3(b)]{An90}.

Regarding Riemannian manifolds with
almost nonnegative Ricci curvature, the following is known.

\begin{thm} \label{extra}
Given $n \in \Z^+$, there are $\eps = \eps(n) >0$ and $J = J(n) < \infty$
with the following property.
Suppose that $(M,g)$ is a closed connected Riemannian $n$-manifold with
$\diam(M)\le 1$ and $\Ric_M \ge - \eps$.
\begin{enumerate}
\item \label{C-nilp} Then $\pi_1(M)$ has a nilpotent subgroup (of 
index at most $J$) with  nilpotency rank at most $n$ ~\cite{KW11}.
(This was also proved in \cite{BGT12} without the index bound.)
\item 
If $M$ is spin and $n$ is divisible by four then 
$|\hat{A}(M)| \le 2^{\frac{n}{2}-1}$
\cite{Ga83}, \cite[p. 86]{Gr82}.
\end{enumerate}
\end{thm}

Theorem \ref{thm-lower-ricci-noncol}  refines
Theorem \ref{extra}(1) in the noncollapsed case,
although the constants $I$ and $J$ may not be related.

\begin{rmk}
In four dimensions, Rokhlin's theorem says that $\hat{A}(M)$ is
even if $M$ is spin.  Hence under the assumptions of
Theorem \ref{extra}(2), if $n=4$ and $\widehat{A}(M) \neq 0$ then
$|\widehat{A}(M)| = 2$. This means that after possibly reversing
orientation, $M$ is spin-cobordant to the $K3$ surface. 
One can ask whether Theorem \ref{main-thm} still holds without the
upper Ricci curvature bound, say assuming a 
lower Ricci curvature bound and possibly a lower volume bound.

In a somewhat different direction, Cabezas-Rivas and Wilking answered
a question from \cite{Lo00} by showing that a spin manifold with
almost nonnegative sectional curvature has vanishing $\widehat{A}$-genus
\cite{CW}.
\end{rmk}

Regarding almost Ricci-flat $4$-manifolds, Brendle and Kapouleas
found an obstruction to perturbing an
almost Ricci-flat metric to a Ricci-flat metric
\cite{BK}. In their example, the relevant limit space was a flat
$4$-dimensional orbifold. A gluing obstruction for a
nonflat limit space was found by Biquard \cite{Bi13}.

Finally, one can ask whether 
Corollary \ref{cor-large-pi1}(1) holds without
the upper Ricci curvature bound.

\subsection{Structure of the paper}

In Section \ref{sect1} we prove
Theorem~\ref{main-thm}. In Section \ref{sect2} we prove 
Theorem~\ref{main-large-pi1}
and Corollary~\ref{cor-large-pi1}. 
In Section \ref{sect3} we prove Theorem~\ref{thm-lower-ricci-noncol}.

We thank the participants of the MSRI Spring 2016 geometry program for helpful
discussions, especially Olivier Biquard, Ronan Conlon and Jeff Viaclovsky.
We also thank Jeff for comments on an earlier version of the paper.
We thank the referee for pointing out  \cite{BGR07}.

\section{Proof of Theorem ~\ref{main-thm}} \label{sect1}

Arguing by contradiction, if Theorem ~\ref{main-thm} is not true then there is
a sequence $\{(M_i, g_i)\}_{i=1}^\infty$ of closed connected Riemannian spin $4$-manifolds with
\begin{itemize}
\item $\widehat{A}(M_i) \neq 0$,
\item $\vol(M_i) \ge v$,
\item $\diam(M_i) \le 1$,
\item $\Ric_{M_i} \le \Upsilon g_i$ and
\item $S_{M_i} \ge - \frac{1}{i}$, but
\item $M_i$ is not diffeomorphic to a $K3$ surface.
\end{itemize}

\subsection{Harmonic spinors on $M_i$} \label{subsect1}
By assumption, the second Stiefel-Whitney class of $M_i$ vanishes in
$H^2(M_i; \Z_2)$.
After reversing orientation if necessary, we can assume that 
$\widehat{A}(M_i) > 0$. We can then choose a spin structure on $M_i$
that is compatible
with this orientation.
The Atiyah-Singer index theorem implies that 
$M_i$ has a nonzero harmonic spinor $\phi_i$ of positive chirality. 
By rescaling, we can
assume that $\| \phi_i \|_2 = 1$. Letting $D$ denote the Dirac operator,
the Lichnerowicz formula gives
\begin{equation} \label{2.1}
0 = \int_{M_i} |D\phi_i|^2 \: \dvol_{M_i} =  \int_{M_i} 
\left( |\nabla \phi_i|^2+\frac{S_{M_i}}{4} |\phi_i|^2 \right) \: \dvol_{M_i}.
\end{equation}
Hence 
\begin{equation} \label{2.2}
 \int_{M_i} |\nabla \phi_i|^2 \: \dvol_{M_i} \le \frac{1}{4i}.
\end{equation}
As 
\begin{equation} \label{2.3}
\nabla |\phi_i| = \nabla \sqrt{\langle \phi_i, \phi_i \rangle} =
\frac{
\langle \nabla \phi_i, \phi_i \rangle + \langle \phi_i,  \nabla \phi_i \rangle
}{
2 |\phi_i| 
}
\end{equation}
away from the zero-locus of $\phi_i$, we have
\begin{equation} \label{2.4}
|\nabla |\phi_i|| \le |\nabla \phi_i|
\end{equation}
and so 
\begin{equation} \label{2.5}
\int_{M_i} |\nabla |\phi_i||^2 \: \dvol_{M_i} \le \frac{1}{4i}.
\end{equation}

Let $\lambda_i > 0$ denote the Poincar\'e constant of $(M_i, g_i)$, so that
\begin{equation} \label{2.6}
\int_{M_i} |\nabla F|^2 \: \dvol_{M_i} \ge \lambda_i
\int_{M_i} (F - \overline{F})^2 \: \dvol_{M_i}
\end{equation}
for all $F \in H^1(M_i)$, where $\overline{F}$ denotes the average value
$\overline{F} = \frac{1}{\vol(M_i)} \int_{M_i} F \: \dvol_{M_i}$.
Then
\begin{equation} \label{2.7}
\int_{M_i} \left( |\phi_i| - \overline{|\phi_i|} \right)^2 \: \dvol_{M_i} 
\le \frac{1}{4 i \lambda_i}.
\end{equation}

\begin{lem} \label{2.8}
If $V_i$ is open in $M_i$ then
\begin{equation} \label{2.9}
\int_{V_i} |\phi_i|^2 \: \dvol_{M_i} \le
\left(
\frac{1}{ \sqrt{4 i \lambda_i}}
 + \sqrt{ \frac{\vol(V_i)}
{\vol(M_i)}} \right)^2.
\end{equation}
\end{lem}
\begin{proof}
We work more generally with a measure space $(M, d\mu)$ having finite mass,
a measurable subset $V \subset M$ and a function $f \in L^2(M, d\mu)$ with
$\| f \|_{L^2(M)} = 1$. By the Cauchy-Schwarz inequality,
$| \overline{f} | \le \frac{1}{\sqrt{\mu(M)}}$. Then
\begin{equation} \label{2.10}
\|f\|_{L^2(V)} \le \|f - \overline{f} \|_{L^2(V)} + \|\overline{f}\|_{L^2(V)}
\le  \|f - \overline{f} \|_{L^2(M)} + \sqrt{\frac{\mu(V)}{\mu(M)}},
\end{equation}
from which the lemma follows.
\end{proof}

\subsection{Parallel spinors on $X$} \label{subsect2}

From the scalar curvature condition and the upper bound on Ricci curvature,
for large $i$ we have $\Ric_{M_i} \ge - 10 \Upsilon g_i$.
Then from \cite[Theorem 1.13]{Ch-Na-codim4},
there is a uniform upper bound on
$\int_{M_i} |\Rm_{M_i}|^2 \: \dvol_{M_i}$. 
By \cite[Theorem 2.6 and Pf. of Main Lemma 2.2]{An90}, 
after passing to a subsequence we can assume that 
$\lim_{i \rightarrow \infty} (M_i, g_i) = (X, g_X)$ in the Gromov-Hausdorff topology, where
{
\begin{itemize}
\item $X$ is a four dimensional compact orbifold with
finitely many isolated orbifold singularities, 
\item $g_X$ is a continuous orbifold Riemannian metric on $X$, and
\item $g_X$ is locally $W^{2,p}$-regular away from the singular points, 
for all $p < \infty$.
\end{itemize}
Let $X_{reg}$ denote the regular part of the orbifold $X$, i.e.
the complement of the finitely many singular points $X_{sing}$.
It is a $W^{3,p}$-manifold, for all $p < \infty$, so we can assume that it is
equipped with an underlying smooth structure (although with a
Riemannian metric that is $W^{2,p}$-regular in that smooth
structure). Then $X_{reg}$ is also a $C^{2,\alpha}$-manifold  and $g_X$ is
locally $C^{1,\alpha}$-regular, for all $\alpha \in (0,1)$.
}

For large $j$,
let $U_j$ be the union of the $\frac{1}{j}$-balls around the singular
points in $X$. Note that $X- U_j$ and $M_i$ are both smooth.
For large $i$, there is a smooth map
$\sigma_{j,i} : (X - U_j) \rightarrow M_i$ that is a diffeomorphism onto its
image, so that $\lim_{i \rightarrow \infty} \sigma_{j,i}^* g_i = g_{X- U_j}$
in the $C^{1,\alpha}$-topology. In particular, 
$X - U_j$ admits a spin structure. 
As $H^1(X - U_j; \Z_2)$ is finite,
after passing to a subsequence of $i$'s, we can assume that
each $\phi_{j,i}$ is spin compatible.
From (\ref{2.2}),
\begin{equation} \label{2.11}
\int_{\sigma_{j,i}(X-U_j)} |\nabla \phi_i|^2 \: \dvol_{M_i} \le \frac{1}{4i}.
\end{equation}
From Lemma \ref{2.8},
\begin{align} \label{2.12}
& \int_{\sigma_{j,i}(X-U_j)} |\phi_i|^2 \: \dvol_{M_i} \ge \\
& 1 - \left(
\frac{1}{ \sqrt{4 i \lambda_i}}
 + \sqrt{ \frac{\vol(M_i - \sigma_{j,i}(X-U_j))}
{\vol(M_i)}} \right)^2. \notag
\end{align}

It makes sense to compare spinor fields on
two diffeomorphic Riemannian manifolds (c.f. \cite[p. 531-532]{Lo00}), 
so we can consider $\sigma_{j,i}^* \phi_i$, a spinor field on $X - U_j$.
The $H^1$-norm on spinor fields over $X-U_j$ is
\begin{equation} \label{2.13}
\| \psi \|^2_{H^1(X - U_j)} =
\int_{X-U_j} \left( |\psi|^2 + |\nabla \psi|^2 \right) \: \dvol_{X-U_j}.
\end{equation}
{
Note that the Christoffel symbols on $X- U_j$ are locally $C^\alpha$-regular
and locally $W^{1,p}$-regular.
}
From (\ref{2.11}) and the normalization of $\phi_i$, the $H^1$-norms of
$\{ \sigma_{j,i}^* \phi_i \}_{i=1}^\infty$ are uniformly bounded, so
we can take a subsequence that converges weakly in $H^1$ to some
positive chirality
spinor field $\psi_j$ on $X - U_j$. By Rellich compactness,
after passing to a further subsequence
we can assume that $\lim_{i \rightarrow \infty} \sigma_{j,i}^* \phi_i
= \psi_j$ in $L^2$. In particular,
\begin{align} \label{2.14}
\|\psi_j\|^2_{L^2(X-U_j)} = & 
\lim_{i \rightarrow \infty} \| \sigma_{j,i}^* \phi_i \|^2_{L^2(X-U_j)} 
= \lim_{i \rightarrow \infty} \| \phi_i \|^2_{L^2(\sigma_{j,i}(X-U_j))} \\
= & \lim_{i \rightarrow \infty} \int_{\sigma_{j,i}(X-U_j)} |\phi_i|^2 \:
\dvol_{M_i}. \notag
\end{align}
As norms can only decrease when taking weak limits, using (\ref{2.11}) we have
\begin{align} \label{2.15}
\|\psi_j\|^2_{H^1(X-U_j)} \le &
\liminf_{i \rightarrow \infty} \| \sigma_{j,i}^* \phi_i \|^2_{H^1(X-U_j)} 
= \liminf_{i \rightarrow \infty} \| \phi_i \|^2_{H^1(\sigma_{j,i}(X-U_j))} \\
= & \liminf_{i \rightarrow \infty} \int_{\sigma_{j,i}(X-U_j)} \left(
|\nabla \phi_i|^2 + |\phi_i|^2 \right) \:
\dvol_{M_i} \notag \\
= & \lim_{i \rightarrow \infty} \int_{\sigma_{j,i}(X-U_j)} |\phi_i|^2 \:
\dvol_{M_i}. \notag
\end{align}
Thus $\|\psi_j\|_{H^1(X-U_j)} = \|\psi_j\|_{L^2(X-U_j)}$, so
$\nabla \psi_j$ vanishes weakly. 
 
There is a uniform positive lower bound on $\lambda_i$ in terms of the
upper diameter bound and the lower Ricci bound; see \cite{BQ00} and
references therein. We have convergence 
$\lim_{i \rightarrow \infty} (M_i, g_{M_i}, \dvol_{M_i}) \rightarrow
(X, g_X, \dvol_X)$ in the measured Gromov-Hausdorff topology. Then
using (\ref{2.12}), we find
\begin{equation} \label{2.16}
1 - \frac{\vol(U_j)}{\vol(X)} \le \|\psi_j\|^2_{L^2(X-U_j)} \le 1. 
\end{equation}

The preceding construction of $\psi_j$ was for a fixed but sufficiently
large $j$. 
We can take $j \rightarrow \infty$, and apply a diagonal
argument in $j$ and $i$,
to obtain a weakly parallel 
positive chirality
spinor field $\psi_\infty$ on the spin manifold $X_{reg}$,
with $\| \psi_\infty \|_{L^2(X_{reg})} = 1$.

In a coordinate chart and using an orthonormal frame $\{e_a\}_{a=1}^4$, 
the fact that $\psi_\infty$ is weakly parallel means that
\begin{equation} \label{parallel}
\partial_k \psi_\infty = -  \frac{1}{8} \sum_{a,b=1}^4
\Gamma_{abk} [\gamma^a, \gamma^b] \psi_\infty
\end{equation}
holds weakly, where $\{\gamma^a \}_{a=1}^4$ are the Dirac matrices.
We know that $\psi_\infty$ is $W^{1,2}$-regular.
As $\Gamma$ is $W^{1,p}$-regular for all $p < \infty$, using
the Sobolev embedding theorem and bootstrapping, one finds that
$\psi_\infty$ is locally $W^{2,p}$-regular for all $p < \infty$. In particular,
$\psi_\infty$ is locally $C^{1,\alpha}$-regular for all $\alpha \in (0,1)$.
Hence $\psi_\infty$ satisfies (\ref{parallel}) in the classical sense.
Then for any smooth vector field $V$ on $X_{reg}$, we have
\begin{equation}
V \langle \psi_\infty, \psi_\infty \rangle =
\langle \nabla_V \psi_\infty, \psi_\infty \rangle + 
\langle \psi_\infty, \nabla_V \psi_\infty \rangle = 0,  
\end{equation}
showing that $|\psi_\infty|$ is a (nonzero) constant.

Since $\psi_\infty$ is parallel and locally $W^{2,p}$-regular,
for smooth vector fields $V$ and $W$ on $X_{reg}$ we have
\begin{equation}
0 = \nabla_V \nabla_W \psi_\infty - \nabla_W \nabla_V \psi_\infty -
\nabla_{[V,W]} \psi_\infty = R(V,W) \psi_\infty 
\end{equation}
in $L_{loc}^p$, for all $p < \infty$. 
Along with the nowhere-vanishing of $\psi_\infty$, 
this implies algebraically that $\Ric_X = 0$
\cite[Corollary 2.8]{BHMMM15}.

We recall that the $W^{2,p}$-regularity of $g_X$ around $x \in X_{reg}$ 
is derived using coordinates that are constructed by starting with
harmonic coordinates around points $m_i \in M_i$,
with $\lim_{i \rightarrow \infty} m_i = x$ and passing to the limit
\cite[Pf. of Main Lemma 2.2]{An90}. Hence
the formula for $\Ric$ in harmonic coordinates
\cite[(2.7)]{An90} still holds weakly in the
coordinates around $x \in X$. By elliptic regularity, the
vanishing of $\Ric_X$ implies that $g_X$ is smooth on $X_{reg}$,
relative to the smooth structure defined by these coordinates.
Being parallel, $\psi_\infty$ is also smooth.

From \cite[Theorem 1]{BGR07},
the spin structure on $X_{reg}$ extends to an orbifold
spin structure on $X$. 
By removable singularity results for Einstein metrics, $g_X$ is a smooth Ricci-flat
orbifold Riemannian metric on $X$; c.f. ~\cite[\S 5]{BKN89}.

Given 
$p \in X_{sing}$, let $(U, q, \Gamma_p)$ be an orbifold
chart for a neighborhood of $p$. Here $U$ is a ball in $\R^4$ around the
origin $q$, and $\Gamma_p$ is the local group of $p$. The lift of
$\psi_\infty$ is a parallel spinor field $\hat{\psi}_\infty$ on $U-q$. 
Since $U-q$ is simply connected, $\hat{\psi}_\infty$ has a unique extension
to a smooth parallel spinor field over $U$. Hence $\psi_\infty$ extends
uniquely to a nonzero 
positive chirality
parallel spinor field on $X$. 
Given  $x \in X_{reg}$, let $\gamma$ be a special loop at $x$ in the sense of
\cite[Chapter 2.2]{KL14}. Identifying the oriented isometry group of
$T_x X$ with $SO(4)$, from the decomposition $Spin(4) = SU(2) \times SU(2)$
it follows that the holonomies around such $\gamma$'s
lie in one of the $SU(2)$-factors.
(As the parallel spinor field has positive chirality, our conventions are
that the $SU(2)$-factor is the second factor.)
That is, $X$ acquires the structure of a hyperK\"ahler orbifold.
In particular, if $p \in X_{sing}$ then we can take the orbifold chart
$(U, q, \Gamma_p)$ to have $U$ an open ball in $\C^2$ with origin $q$ and
$\Gamma_p$ a finite subgroup of $SU(2)$ that acts freely on $T_q^1 U
\cong S^3$.

\subsection{ALE blowups} \label{subsect3}

If $p \in X_{sing}$ then there are points $p_i \in M_i$ so that
$\lim_{i \rightarrow \infty} (M_i, p_i) = (X, p)$ in the pointed
Gromov-Hausdorff topology. After passing to a subsequence, from
\cite[Remark 3.1]{AC91} and \cite[Proposition 2]{Bando90}
there is an appropriate sequence $\{\delta_i\}_{i=1}^\infty$ 
of positive numbers with
$\lim_{i \rightarrow \infty} \delta_i = 0$ so that
$\lim_{i \rightarrow \infty} \left( M_i, \frac{1}{\delta_i^2} g_{M_i}, 
p_i \right) = (Y, g_Y, y_0)$ in the pointed Gromov-Hausdorff topology,
where $Y$ is an nonflat Ricci-flat
ALE orbifold with finitely many orbifold singular points. 
The decay rate of $Y$ is order-$4$ in the terminology of
\cite{BKN89}; see \cite[Proposition 2]{Bando90} and \cite[Theorem 1.5]{BKN89}.
For any small $\epsilon > 0$, the complement $C_\epsilon$ of the 
$\epsilon$-neighborhood of $Y_{sing}$ embeds in $M_i$ for large $i$.
In particular, $C_\epsilon$ admits a spin structure.  Since 
$H^1(C_\epsilon; \Z_2)$ is finite, after passing to a subsequence 
we can assume that the embeddings are spin compatible. Taking $\epsilon$ going 
to zero, $Y_{reg}$ acquires a spin structure.
From \cite[Theorem 1]{BGR07},
the orbifold $Y$ is a spin orbifold. 
Its tangent cone at infinity is $\C^2/\Gamma_p$.

\begin{lem} \label{2.17}
$Y$ has a nonzero parallel spinor field.
\end{lem}
\begin{proof}
The proof is essentially the same as that of \cite[Corollary 3.4]{Na90}, which
treats the case when $Y$ is a manifold, with only minor changes.
To make this clear, we outline the steps of the proof.
From the existence of the positive chirality parallel spinor 
$\psi_\infty$ in a neighborhood of $p \in X$, the action of $\Gamma_p$ on
$S^3$ is a right action in the sense of \cite[p. 390]{Na90}.
As in Witten's proof of the positive mass theorem, one constructs
a positive chirality
spinor field $\eta_0$ on $Y$ that is asymptotically parallel at infinity,
with norm approaching one.  Then one puts $\eta = \eta_0 -
D G D \eta$, where $D$ is the Dirac operator and $G$ is the 
inverse of the (invertible) operator $D^2$, when considered as an operator
between appropriate weighted function spaces. The spinor field
$\eta$ is harmonic and has the same asymptotics as $\eta_0$.
This linear analysis extends with only trivial change to the
orbifold case.  

Integration by parts, and the fact that $Y$ is scalar-flat, give
\begin{equation}
m = \const \int_Y |\nabla \eta|^2 \: \dvol_Y,
\end{equation}
where $m$ is the ADM mass, as defined using integration over large
distance spheres in $Y$.  The order-$4$ decay rate of $Y$ implies that
$m = 0$. Hence 
the positive chirality spinor field
$\eta$ is parallel.

\end{proof}

Thus $Y$ is also a hyperK\"ahler orbifold. 

\subsection{Bubble tree} \label{subsect4}

We can repeat the blowup analysis
at a point of $Y_{sing}$,
using the fact that $Y$ has a nonzero parallel spinor field
of positive chirality.
The result is that we get a bubble tree $T$,
as in \cite{AC91,Bando90}. (The papers \cite{Bando90,Bando90-2} treat
the case when the initial manifold is Einstein, while \cite{AC91} treats
the more general case of bounded Ricci curvature.) This is
a finite directed rooted tree, 
with a connected orbifold associated to each vertex. 
The orbifold associated to the root vertex is $X$. The orbifolds associated
to the other vertices $v$ are Ricci-flat ALE orbifolds $W_v$. 
The edges of $T$ 
point inward toward the root vertex. Given a
vertex $v \in T$, the edges with terminus $v$ are in
bijective correspondence with $W_{v,sing}$. The initial vertex $v^\prime$
of such
an edge $e$ is the result of the blowup analysis at the corresponding 
point $w \in W_{v,sing}$. If $\Gamma_{w}$ is the local group of
$w$ then the asymptotic cone of $W_{v^\prime}$ is $\C^2/\Gamma_w$.
The finiteness of $T$ comes from the uniform upper bound on
$\int_{M_i} |\Rm_{M_i}|^2 \: \dvol_{M_i}$, since each blowup orbifold has
a definite amount of $\int |\Rm|^2 \: \dvol$.

Given $k \ge 0$, let $T_k$ be the vertices of distance $k$ from the
root vertex.  Let $N$ be the largest $k$ for which $T_k \neq \emptyset$,
which we will call the height of $T$.
The orbifolds associated to vertices in $T_N$ are manifolds.

From $T$ we can construct a compact smooth manifold $M_T$ that is
diffeomorphic to $M_i$, for large $i$. To describe $M_T$, first
consider the case when $N=0$. 
Then $X$ is a smooth manifold and $M_T = X$.
If $N=1$ then $T$ consists of the root vertex along with 
vertices in $T_1$. The only edges in $T$ 
join vertices in
$T_1$ to the root vertex. 
Given $x \in X_{sing}$, a small neighborhood $O_x$ of $x$ is 
orbifold-diffeomorphic to a finite cone over the space form
$S^3/\Gamma_x$. If $v^\prime \in T_1$
is the initial vertex of the edge corresponding to $x$ then  its associated
orbifold $W_{v^\prime}$ 
has asymptotic cone $\C^2/\Gamma_x$. Let $P_x$ be a truncation
of $W_{v^\prime}$
whose boundary is a copy of $S^3/\Gamma_x$ at large distance.
We remove $O_x$ from $X$ and glue in a rescaled copy of $P_x$.
Doing this for all $x \in X_{sing}$ gives $M_T$.
Given the combinatorics of the matchings,
the gluing process is unique up to isotopy,
because of the existence of the small-scale asymptotics near $x$ and
the large-scale asymptotics of $W_{v^\prime}$.

For $N > 1$, we do the same gluing procedure inductively.  That is,
given a vertex $v \in T_{N-1}$, if there are no edges terminating at
$v$ then the orbifold associated to $v$ is a manifold and we leave it alone.
If there are edges terminating at $v$ then the orbifold associated to
$v$ has singular points.
We remove neighborhoods of the singular
points and glue in truncated manifolds associated to the vertices 
$v^\prime \in T_N$ that are joined to $v$ in $T$. Doing this for all
$v \in T_{N-1}$, the result is a tree of height $N-1$ for which
the orbifolds associated to the vertices of distance $N-1$ from the
root vertex are all noncompact manifolds.
Then we iterate downward in the height of the tree, until
we finish with $M_T$.

From the uniqueness of the gluing procedure, up to isotopy, if we
know the orbifolds associated to the vertices of $T$
and the combinatorics of the matchings
then we
can uniquely determine the diffeomorphism type of $M_T$.

\subsection{End of the proof} \label{subsect5}

In the previous subsection, we did not make reference to the hyperK\"ahler
structures.  In our situation, all of the orbifolds associated to the
vertices of $T$ carry hyperK\"ahler structures.

\begin{lem} \label{2.18}
For large $i$, the manifold $M_i$ is diffeomorphic to a hyperK\"ahler
manifold.
\end{lem}
\begin{proof}
Consider the height $N$ of the bubble tree $T$.  If $N=0$ then $M_i$ is
diffeomorphic to the hyperK\"ahler manifold $X$.  If $N > 0$ then
the orbifolds associated to the vertices of $T$ are hyperK\"ahler ALE
orbifolds. Those associated to vertices in $T_N$ are hyperK\"ahler ALE
manifolds.

For a hyperK\"ahler ALE manifold associated to a vertex in $T_N$,
if its asymptotic cone is $\C^2/\Gamma$ then the manifold 
is deformation equivalent to the minimal resolution of $\C^2/\Gamma$, 
through hyperK\"ahler ALE manifolds with asymptotic cone $\C^2/\Gamma$
\cite[Theorem 7.2.3]{Jo00},\cite{Kr89}.
Consequently, for our gluing purposes, 
we can assume that the hyperK\"ahler structure
on the ALE manifold is exactly that of the minimal resolution.

Consider first the case $N=1$. 
Using the compatible trivializations of the positive chirality spinor bundles,
in a neighborhood of a singular point of $X$ and at the infinity of the
corresponding ALE manifold, there is no ambiguity in the matchings.
We see that
$M_T$ is the minimal
resolution of the hyperK\"ahler orbifold $X$.
As $X$ is hyperK\"ahler,
it has a trivial canonical bundle.  The minimal resolution of an
orbifold of complex dimension two
is a crepant resolution, i.e. the minimal resolution of $X$ also has
a trivial canonical bundle. Hence it admits a hyperK\"ahler structure.

If $N > 1$ then 
let $v$ be a vertex in $T_{N-1}$, with associated hyperK\"ahler ALE
orbifold $W_v$. If there are no edges terminating at $v$ then $W_v$ is
a hyperK\"ahler ALE manifold. 
As before, we deform it to a minimal resolution of a $\C^2/\Gamma$.
If there
are edges terminating at $v$ then $W_v$ is a hyperK\"ahler ALE orbifold
with singular points.  From \cite[Theorem 4]{Bando90}, after we 
perform the gluing procedure around the orbifold singular points of
$W_v$, the result has the structure of a hyperK\"ahler ALE manifold.
We again deform it to a minimal resolution of a $\C^2/\Gamma$.

Doing this for all $v \in T_{N-1}$, we have reduced to a tree of
height $N-1$. The lemma follows from downward iteration.
\end{proof}

Lemma \ref{2.18} says that for large $i$, the manifold $M_i$ is diffeomorphic
to a compact hyperK\"ahler $4$-manifold. 
Such a hyperK\"ahler manifold is a
$4$-torus or a $K3$ surface
\cite[Chapter 6]{BHPV04}.
Since $\widehat{A}(T^4) = 0$, in either case
we obtain a contradiction to the assumptions of the argument.

\begin{rmk} \label{rmk1}
From the gluing procedure, one obtains 
geometric approximations for the $M_i$'s, 
for large $i$. It seems possible that one could perturb the geometric
approximation, as in \cite[Theorem 2.5]{BM11},
 in order to find a Ricci-flat metric on $M_i$ that is 
biLipschitz close to $g_i$.
\end{rmk}

\begin{rmk}
The direct higher dimensional analog of 
Theorem ~\ref{main-thm} would be to ask
whether for any $n \in \Z^+$, $v > 0$ and $\Lambda, \Upsilon < \infty$, 
there is some
$\epsilon = \epsilon(n, v, \Lambda, \Upsilon) > 0$ so that if $M$ is a
closed connected spin Riemannian $n$-manifold 
with $\widehat{A}(M) \neq 0$, $\vol(M) \ge v$, $\diam(M) \le 1$, 
$\int_M |\Rm|^{\frac{n}{2}} \: \dvol_M \le \Lambda$,
$\Ric_M \le \Upsilon g_M$ and
$S_M \ge - \epsilon$,
then $M$ admits a Ricci-flat metric of special holonomy.

The discussions of Subsections \ref{subsect3}-\ref{subsect4} 
go through without change to
produce a limit orbifold $X$ with special holonomy, blowup 
Ricci-flat ALE orbifolds
with special holonomy and a bubble tree.  However, if $n > 4$ then
there are more possibilities for the holonomies and hence more
possibilities for the orbifolds associated to the vertices of the
bubble tree.  Rather than trying to classify the possibilities, it
is conceivable that one could perform a gluing construction, as 
mentioned in
Remark \ref{rmk1}, in order to directly construct Ricci-flat metrics
of special holonomy.  
\end{rmk}

\section{Proofs of Theorem ~\ref{main-large-pi1} and Corollary ~\ref{cor-large-pi1}} \label{sect2}

We first prove Theorem ~\ref{main-large-pi1}(1).
Arguing by contradiction, if it is not true then for some 
$n \in \Z^+$,  $v > 0$ and $\Lambda, \Upsilon < \infty$,
there is a sequence $\{(M_i,g_i)\}_{i=1}^\infty$ of closed connected $n$-dimensional
Riemannian manifolds with
\begin{itemize}
\item $\int_{M_i} | \Rm_{M_i} |^{\frac{n}{2}} \: \dvol_{M_i} \le \Lambda$,
\item $\vol(M_i) \ge v$,
\item $\diam(M_i)\le 1$,
\item $- \: \frac{1}{i} g_i \le \Ric_{M_i} \le \Upsilon g_i$ and
\item $|\pi_1(M_i)|\ge i$, but
\item $M_i$ does not admit a $W^{2,p}$-regular Riemannian metric 
$h_i$ with nonnegative measurable Ricci curvature for which the universal cover
$(\tilde{M}_i, \tilde{h}_i)$ isometrically splits off an $\R$-factor.
\end{itemize}

%Therefore, by  ~\cite[Theorem 1.12]{Ch-Na-codim4} or ~\cite{AC91}, $M_i$ fall into at most finitely many diffeomorphism classes.
By \cite[Theorem 2.6]{An90}, 
after passing to a subsequence we can assume that 
$\lim_{i \rightarrow \infty} (M_i, g_i) = (X, g_X)$ in the Gromov-Hausdorff topology, where
\begin{itemize}
\item $X$ is an $n$-dimensional compact orbifold with
finitely many isolated orbifold singularities, 
\item $g_X$ is a continuous orbifold Riemannian metric on $X$, and
\item $g_X$ is locally $C^{1,\alpha}$-regular away from the singular points, for all
$\alpha \in (0,1)$.
\end{itemize}
As before,
$g_X$ is locally $W^{2,p}$-regular on $X_{reg}$, for all
$p < \infty$. In particular, $g_X$ has a locally-$L^p$ Ricci tensor.

%Since the sequence $M_i$ only contains finitely many diffeomorphism types, there are only finitely many possible isomorphism types for $\pi_1(M_i)$. Therefore, by again passing to a subsequence we can assume that $|\pi_1(M_i)|=\infty$ for all $i$.

Given $x \in X$, there are points $p_i \in M_i$ so that
$\lim_{i \rightarrow \infty} (M_i,p_i) = (X,x)$ in the pointed Gromov-Hausdorff topology.
Put $\Gamma_i=\pi_1(M_i, p_i)$, let $\pi_i\co \tilde M_i\to M_i$ be the universal cover of $M_i$
equipped with the pullback metric $g_{\tilde M_i}$, and 
pick $\tilde p_i\in \pi_i^{-1}(p_i)$.
After passing to a subsequence, we can assume that $\lim_{i \rightarrow \infty} (\tilde M_i, \tilde p_i,\Gamma_i) =
(Y, \tilde p, \Gamma)$ in the equivariant pointed Gromov-Hausdorff topology of
\cite[\S 3]{FY92}, where
$(Y, \tilde p)$ is a pointed length space  on which $\Gamma$ acts by isometries, with quotient $X$.

Since $|\pi_1(M_i)|\ge i$, we know that $\vol (\tilde M_i) \ge iv$, so 
$\lim_{i \rightarrow \infty} \vol (\tilde M_i) = \infty$.
Using the fact that $\Ric_{M_i}\ge -1$, volume comparison implies that  
$\lim_{i \rightarrow \infty} \diam (\tilde M_i) = \infty$.  Hence 
$\diam Y=\infty$ and $Y$ is noncompact. 
We have $\Ric_{M_i}\ge - \: \frac{1}{i}$, 
so the Cheeger-Colding almost splitting theorem~\cite[Theorem 6.64]{CC} holds on $Y$.
As the action of $\Gamma$ on $Y$ is cocompact, the Cheeger-Gromoll argument 
\cite[\S 9]{CG72} implies that $Y$ is isometric to 
$\R^m\times Z$, where $1 \le m\le n$ and $Z$ is a compact
length space.

\begin{lem} \label{rball}
Given $r > 0$, there is a upper bound on 
the integral of $|\Rm_{\tilde M_i}|^{\frac{n}{2}}$
over $r$-balls in $\tilde M_i$, uniform in $i$.
\end{lem}
\begin{proof} 
Since $\Ric_{M_i} \ge - \: \frac{1}{i} g_i$,
Bishop-Gromov relative volume comparison
gives an explicit $c = c(n,v,r) > 0$ such that for any 
$m_i \in  M_i$, we have  $\vol(B_r(m_i))\ge c$.
Let $\tilde m_i$ be a preimage of $m_i$ in $\tilde M_i$. 
Absolute volume comparison gives an
explicit $c^\prime = c^\prime(n,r) < \infty$ such that 
$\vol(B_{10r}(\tilde m_i))\le c^\prime$.
Suppose now that some $m^\prime_i \in B_r(m_i)$ has $N$ preimages in 
$B_r(\tilde m_i)$. Then any $m^{\prime \prime}_i \in B_r(m_i)$
has at least $N$ preimages in $B_{10r}(\tilde m_i)$. Thus 
\begin{equation}
N \le \frac{\vol(B_{10r}(\tilde m_i))}{\vol(B_r(m_i))} \le
\frac{c^\prime}{c}.
\end{equation}
Hence
\begin{equation}
\int_{B_r(\tilde m_i)} |\Rm_{\tilde M_i}|^{\frac{n}{2}} \: 
\dvol_{\tilde M_i} \le
\frac{c^\prime}{c} \int_{M_i} |\Rm_{M_i}|^{\frac{n}{2}} \: 
\dvol_{M_i} \le \frac{c^\prime}{c} \Lambda.
\end{equation}
This proves the lemma.
\end{proof}

Using Lemma \ref{rball},
we  can apply ~\cite{An90} to 
conclude that $Y$ is an $n$-dimensional orbifold with a discrete set of  isolated orbifold singular points,
and a nonnegative measurable Ricci tensor. 
From the splitting $Y\cong\R^m\times Z$, the set of singular points in $Y$ must be empty. 
Then $\lim_{i \rightarrow \infty} (\tilde M_i,\tilde p_i) = (Y,\tilde p)$, with
Riemannian metrics converging in the
pointed weak $W^{2,p}$-topology; c.f. \cite[Remark 2.7(ii)]{An90}.
In particular, $Y$ has nonnegative measurable Ricci curvature.

%And  since all $\tilde M_i$ are simply connected this implies that $Y$ is also simply connected and therefore $Y\cong \R^4$.
Since $X=Y/\Gamma$  is $n$-dimensional, it follows that $\Gamma\subset \Iso(Y)$ is discrete.
{\it A priori}, $\Gamma$ need not act freely on $Y$.  However, since the orbifold singular points of $X$ are
isolated, if $\Gamma$ does not act freely on $Y$ then the points in $Y$ with nontrivial isotropy groups must
be isolated.
We claim that the action is free in our situation.

\begin{lem} 
The group $\Gamma$ acts freely on $Y$.
\end{lem}
\begin{proof}
Suppose to the contrary  that there is a point $q \in Y$ with a nontrivial isotropy group.
Let $\gamma \in \Gamma$ be a nontrivial element that fixes $q$. Then
$q$ is isolated in the fixed point set of $\gamma$.
Since the action of $\gamma$ near $q$ can be linearized,
there exist  $r > 0$ and  $\delta \in (0, r/100)$ such that 
$B_{s}(q)$ is a topological ball for any $s \in (0, 10r)$, and
$d_Y(\gamma(y),y)>\delta$ for any 
$y \in Y$ with $r/2<d_Y(q,y)<2r$.
 
Let $\gamma_i\in \Gamma_i$ converge to $\gamma$ and 
let $q_i\in \tilde M_i$ converge to $q$. 
Since $\gamma(q)=q$ we have $\lim_{i \rightarrow \infty} d_{\tilde M_i}(\gamma_i(q_i), q_i) = 0$.  Hence
for large $i$,
\begin{equation} \label{3.2}
B_{r-
\frac{\delta}{100}
}(q_i)\subset \gamma_i(B_r(q_i)) \subset B_{r+
\frac{\delta}{100}
}(q_i).
\end{equation}
Also, $d_{\tilde M_i} (\gamma_i(\tilde m_i),\tilde m_i) >
\delta/2$ for any $\tilde m_i \in\tilde M_i$ with $r/2<d_{\tilde M_i}(\tilde m_i, q_i)<2r$. 
Since $\lim_{i \rightarrow \infty} (\tilde M_i,q_i) = (Y,q)$ in the pointed  
$C^{1,\alpha}$-topology, 
we can find a closed topological $4$-disk $D_i\subset \tilde M_i$ which is $\eps_i$-Hausdorff 
close to $B_r(q_i)$ with $\eps_i\to 0$. Its image $\gamma_i(D_i)$ is 
$\eps_i$-Hausdorff close to $\gamma_i(B_r(q_i))$. Then for large $i$, 
equation (\ref{3.2}) implies that
$\gamma_i(D_i)$ is 
$\frac{\delta}{50}$-Hausdorff
close to $B_r(q_i)$.
Using the fact that $B_{2r}(q_i)$ is $C^{1,\alpha}$-close to $B_{2r}(q)\subset Y$, by slightly squeezing
$B_{2r}(q_i)$ inward we can 
find a continuous map $\phi_i\co  
B_{2r}(q_i)\to  \tilde M_i $ such that
\begin{itemize}
\item $\phi_i$ acts as the identity on $B_{r-\delta/10}(q_i)$,
\item $\phi_i$ sends $B_{r+\delta/20}(q_i)$ into $B_{r-\delta/20}(q_i)$, and
\item $d_{\tilde M_i} (\phi_i(\tilde m_i), \tilde m_i) < \frac{\delta}{5}$ for all
$\tilde m_i \in B_{2r}(q_i)$.  
\end{itemize}
Then $\phi_i(\gamma_i(D_i)) \subset D_i$. 

The map $\phi_i\circ\gamma_i\co D_i\to D_i$ is 
continuous. We wish to show that for large $i$, it has no fixed points.  If
$\tilde m_i \in B_{r-\delta/5}(q_i)$ then for large $i$, we have
$\gamma_i(\tilde m_i) \in B_{r-\delta/10}(q_i)$ and so
$(\phi_i\circ\gamma_i)(\tilde m_i) = \gamma_i(\tilde m_i) \neq \tilde m_i$.
If  $\tilde m_i \in D_i$ but $\tilde m_i \notin B_{r-\delta/5}(q_i)$ then
\begin{align}
& d_{\tilde M_i}( (\phi_i \circ \gamma_i)(\tilde m_i), \tilde m_i) \: \ge \\
& d_{\tilde M_i}( \gamma_i(\tilde m_i), \tilde m_i) -
d_{\tilde M_i}( (\phi_i \circ \gamma_i)(\tilde m_i), \gamma_i(\tilde m_i)) \: > \:
\frac{3}{10} \delta. \notag 
\end{align}
Hence $\phi_i\circ\gamma_i$ has no fixed points in $D_i$.
This contradicts the Brouwer fixed point theorem, so
$\Gamma$ must act freely on $Y$.
\end{proof}

We now know that $X$ is a $W^{3,p}$-manifold. Therefore,
from convergence theory,
$M_i$ is $W^{3,p}$-diffeomorphic to $X$ for all large $i$.
Pulling back the metric on $X$ to $M_i$ gives a metric $h_i$ that  contradicts our assumptions. This proves Theorem~\ref{main-large-pi1}(1).

To prove part (2) of the theorem, we replace the upper Ricci curvature bound
in the contradiction argument, by the assumption that
\begin{itemize}
\item $|\Ric_{M_i}| \le \frac{1}{i}$.
\end{itemize}
We construct the orbifold $X$ as before.  Because of the Ricci pinching
of $M_i$, the metric
$g_X$ has vanishing measurable Ricci tensor 
away from the singular points. Then
$g_X$ is smooth and Ricci-flat away from the singular points.
By removable singularity results for Einstein metrics, $g_X$ is a smooth Ricci-flat
orbifold Riemannian metric on $X$; c.f. ~\cite[\S 5]{BKN89}. The rest of the proof proceeds as before.
\qed

We now prove part (1) of Corollary~\ref{cor-large-pi1}. 
From \cite[Theorem 1.13]{Ch-Na-codim4}, for any $v >0$ 
there is a constant $\Lambda = \Lambda(v)>0$ so that 
$\vol(M) \ge v$, $\diam(M) \le 1$ and  $|\Ric_M| \le 3$
imply that
$\int_{M}|\Rm|^2 \: \dvol_M \le \Lambda$. 
Without loss of generality, we can assume that the
constant $\epsilon$ in Theorem~\ref{main-large-pi1} is less than $3$.  
From the proof of Theorem~\ref{main-large-pi1}(1),
we obtain a $W^{2,p}$-regular metric $h$ on $M$ with nonnegative
measurable Ricci curvature for which the universal cover 
$(\tilde M, \tilde h)$ is an isometric product
$\R^m \times K$, where $1 \le m \le n$ and $K$ is a compact length space.
Since isometries take lines to lines, the arguments of 
\cite[\S 9]{CG72} show that
$\Iso(\R^m \times K) \cong \Iso(\R^m)\times \Iso(K)$, 
with the action of $\pi_1(M)$ on 
$\R^m\times K$ being diagonal. 
By~\cite[Theorem 1.21]{CN12}, $\Iso(K)$ is a Lie group. 

Now
$K$ is a compact $W^{3,p}$-manifold of dimension at most three,
equipped with a $W^{2,p}$-regular
Riemannian metric $g_{K}$ having nonnegative measurable Ricci tensor.
Also, $\Iso(K)$ is a compact Lie group $G$ that acts on $K$ by
$W^{3,p}$-diffeomorphisms. We can fix an underlying
smooth structure on $K$ for which $G$ acts by
smooth diffeomorphisms; c.f. \cite[Theorem B]{Pa70}.
(Any two such smooth structures are 
$G$-diffeomorphic; c.f. \cite[Theorem A]{Pa70}.)
We can slightly smooth $g_{K}$ to get a sequence $\{g_k\}_{k=1}^\infty$
of Riemannian metrics on $K$ with diameter uniformly bounded above,
volume uniformly bounded below by a positive number, and
$\Ric_{g_k} \ge - \: \frac{1}{k}$. Furthermore, this smoothing can be done
equivariantly with respect to $G$.
(For example, we could pick a smooth $G$-invariant metric
on $K$ and apply the ensuing heat operator, 
acting on symmetric $2$-tensor
fields, to our $W^{2,p}$-regular Riemannian metric for short time.)
From \cite[Corollary 1.12]{Si12}, there is a smooth Riemannian metric
on $K$ with nonnegative Ricci curvature. 
The construction of this metric,
using Ricci flow, can be done $G$-equivariantly.
Hence we obtain a smooth metric $h_\infty$
on $(\R^m \times K)/\pi_1(M)$
with nonnegative Ricci curvature.  There is a
$C^\infty$-diffeomorphism from $M$ to  $(\R^m \times K)/\pi_1(M)$,
the latter being equipped with the quotient smooth structure. 
Pulling back $h_\infty$ to $M$, part (1) of the corollary
follows. 

The proof of part (2) of the corollary is similar but easier, using
Theorem~\ref{main-large-pi1}(2). In this case,
$K$ is Ricci-flat. Since it has dimension at most three, it is flat.
Hence the metric $h^\prime$ of Theorem~\ref{main-large-pi1}(2) is flat.  
The corollary follows. \qed

\begin{rmk}
In the conclusions of Theorem~\ref{main-large-pi1} and 
Corollary~\ref{cor-large-pi1}, a finite cover of $M$ is diffeomorphic to 
a product $S^1\times N$.
\end{rmk}

\begin{rmk}
With regard to the conclusion of Corollary ~\ref{cor-large-pi1}(2),
we can also say that
the flat metric is $C^{1,\alpha}$-close to the original
metric $g$.
\end{rmk}

\section{Proof of Theorem~\ref{thm-lower-ricci-noncol}} \label{sect3}

\begin{comment}
We will often make use of the following easy algebraic lemma
\begin{lem}\label{gen-subg-bound}
Let $G$ be a group generated by $k$ elements $g_1,\ldots \g_k$. and let $H\le G$ be a subgroup of index $\le l$.
Then $H$ can be generated by $\le C_1(k,l)$ elements of word length $\le C_2(k,l)$ with respect to the generators  $g_1,\ldots g_k$.
\end{lem}
\begin{proof}
It's clearly enough to consider the case when $G$ is free in which case the lemma easily follows by looking at all possible coverings of order $\le l$ of the bouquet of $k$ circles.
\end{proof}
\end{comment}

Arguing by contradiction, if Theorem~\ref{thm-lower-ricci-noncol} 
is not true then there is
a sequence $\{(M_i, g_i)\}_{i=1}^\infty$ of closed connected Riemannian
$n$-manifolds with
\begin{itemize}
\item $\vol(M_i) \ge v$,
\item $\diam(M_i) \le 1$ and
\item $ \Ric_{M_i} \: \ge \: - \: \frac{1}{i}$, but
\item $\pi_1(M_i)$ does not have an abelian subgroup (of index at most $i$)
generated by at most $n$ elements.
\end{itemize}

By Gromov's precompactness theorem, we can assume that 
$\lim_{i \rightarrow \infty} (M_i, g_i) =  (X, d_X)$ in the Gromov-Hausdorff topology,
where $(X, d_X)$ is a compact length space. The uniform lower volume bound on the
$M_i$'s implies that $X$ has Hausdorff dimension $n$
~\cite[Theorem 5.9]{CC97}. 

Let $p\in X$ be a regular point, meaning that every tangent cone $T_pX$ is isometric to 
$\R^n$. The existence of regular points is guaranteed by 
~\cite[Theorem 2.1]{CC97}.  
A small neighborhood of $p$ is homeomorphic 
to an open  ball in $\R^n$ ~\cite[Theorem A.1.8]{CC97}.
Furthermore, for any sequence 
$p_i\in M_i$ converging to $p$, there is some $\eps>0$ such that for all large $i$,
the ball $B_\eps(p_i)$ is contained in a topological disk $D_i\subset M_i$ ~\cite[Theorem A.1.8]{CC97}.

In particular,  for all large $i$,
\begin{equation}\label{short-contract}
\text{any loop at $p_i$ of length at most $\eps/2$ is contractible.}
\end{equation}
By passing to a subsequence, we can assume that this is true for all $i$.

\begin{comment}
Let us recall the following notion due to Gromov.

Given a complete Riemannian  manifold $(M,p)$, let $\Gamma=\pi_1(M)$, $\tilde M\to M$ be the universal cover and let $\tilde p$ be a point in the preimage of $p$.  
For $\gamma\in \Gamma$  we will refer to 
$|\gamma|= d(\tilde p,\gamma(\tilde p))$ as the norm or the length of $\gamma$.
Choose $\gamma_1\in \Gamma$ 
with  the minimal norm in $\Gamma$.
Next choose $\gamma_2$ to have 
minimal norm in $\Gamma\backslash \langle\gamma_1\rangle$. 
On the $n$-th step choose 
$\gamma_n$ to have minimal norm in 
$\Gamma\backslash \langle\gamma_1,\gamma_2,...,\gamma_{n-1}\rangle$.
The sequence $\{\gamma_1,\gamma_2,...\}$ 
is called a {\it short basis} of $\Gamma$ at $\tilde p$. 
In general, the number of elements of a short basis can be finite or infinite.

For any $i>j$ we have $|\gamma_i|\le |\gamma_j^{-1}\gamma_i|$.

If $M$ is a closed manifold, then  the short basis is 
finite and $|\gamma_i|\le 2\diam(M)$.
\end{comment}

As in Section \ref{sect2}, put  $\Gamma_i=\pi_1(M_i, p_i)$, let $\pi_i\co \tilde M_i\to M_i$ be the universal cover of $M_i$
equipped with the pullback metric $g_{\tilde M_i}$, and 
pick $\tilde p_i\in \pi_i^{-1}(p_i)$.
After passing to a subsequence, we can assume that $\lim_{i \rightarrow \infty} (\tilde M_i, \tilde p_i,\Gamma_i) =
(Y, \tilde p, \Gamma)$ in the equivariant pointed Gromov-Hausdorff topology, where $\Gamma$ is a closed subgroup of $\Iso(Y)$ and $Y/\Gamma=X$.

\begin{comment}
By \cite[Theorem 2.5]{KaWil},  $\pi_1(M_i, p_i)$ can be generated by  $N \le C(n)$ loops of length $\le 10$ for all $i$.  Let $\gamma_1^i,\ldots, \gamma_N^i$ corresponding generators of $\Gamma_i$ with $|\gamma_i^j|\le 10$
where by

Let  $|\gamma|:=d(\tilde p,\gamma(\tilde p))$.

\end{comment}

Given $\gamma_i \in \Gamma_i$  we will refer to 
$|\gamma_i|= d(\tilde p_i,\gamma_i(\tilde p_i))$ as the norm or the length of $\gamma_i$. We will use the same notation 
for elements $\gamma$ of $\Gamma$.
By (\ref{short-contract}),  any nontrivial $\gamma_i \in \Gamma_i$ satisfies
$ |\gamma_i|\ge \eps/2$. This property passes to the limit, so
\begin{equation}\label{short-triv}
\text{ any nontrivial $\gamma\in \Gamma$ satisfies $|\gamma|=d(\tilde p,\gamma \tilde p)\ge \eps/2$.}
\end{equation}
 
The length space $Y$ satisfies the splitting theorem
\cite[Theorem 6.64]{CC}.
Since the action of $\Gamma$ on $X$ is cocompact, the Cheeger-Gromoll argument 
\cite[\S 9]{CG72} implies that $Y$ is isometric to $\R^m\times Z$ where $Z$ is a compact
length space and $m\le n$. 
Furthermore, since isometries take lines to lines, the arguments of \cite[\S 9]{CG72} show that
$\Iso(Y)\cong \Iso(\R^m)\times \Iso(Z)$, with the action of $\Gamma$ on $Y=\R^m\times Z$ being diagonal. 
Let $\phi\co \Gamma\to \Iso(\R^m)$ be the composition of inclusion $\Gamma \rightarrow \Iso(Y)$
and projection onto the first factor. Put $H=\ker \phi$. 
Then $H$ is a subgroup of the 
compact group $\Iso(Z)$.
% By~\cite[Theorem 1.21]{CoNa}, $\Iso(Z)$ is a Lie group. 
Property ~\eqref{short-triv} and the compactness of $Z$ imply that $H$ is a discrete
subgroup of $\Iso(Z)$, and hence is finite.
Again using property ~\eqref{short-triv}, we obtain that $L=\phi(\Gamma)$ is a closed discrete subgroup of  
$\Iso(\R^m)$. It must be cocompact and hence it contains a free abelian subgroup of finite index and rank $m$,
consisting only of translations. 
There is a short exact sequence 
\begin{equation} \label{sess}
1\to H\to\Gamma\to L\to 1,
\end{equation} 
where $H$ is finite, $\Gamma$ is finitely presented and $L$ is crystallographic.

Let $\langle \gamma_1,\ldots \gamma_{N_1}|w_1,\ldots, w_{N_2}\rangle$ be a presentation of $\Gamma$. 
By possibly increasing the generating set, we can assume 
that it contains all of the elements $\gamma$ of $\Gamma$ 
with $|\gamma|\le 10$. 
The number of such elements is still bounded because of  
\eqref{short-triv} and Bishop-Gromov volume comparison.
Let $T$ be the maximum wordlength of the relations $w_1,\ldots, w_{N_2}$.

For large $i$, 
properties \eqref{short-contract} and \eqref{short-triv}, along with
the definition of equivariant pointed Gromov-Hausdorff convergence
\cite[\S 3]{FY92}, imply that for any generator $\gamma_j$,
there is a  unique $\gamma_j^i\in\Gamma_i$ that approximates 
$\gamma_j$ in the sense of \cite[Definition 3.3]{FY92}.
In other words, all generators $\gamma_j$ have unique ``lifts'' to $\Gamma_i$.  
In particular, $\gamma_j^i$ will be
uniformly $\frac{\eps}{10T}$-close 
to $\gamma_j$  on $B_{10T}(\tilde p_i)$, in the sense of
\cite[Definition 3.3(4,5)]{FY92}.
Hence any relation $w_s(\gamma_1,\ldots,\gamma_{N_1})=1 $ remains true for the lifts, i.e. 
$w_s(\gamma_1^i,\ldots,\gamma_{N_1}^i)=1$, as otherwise one would obtain a 
noncontractible loop at 
$p_i$ with length at most $\eps/2$.

Thus there is a homomorphism $\rho_i\co \Gamma\to\Gamma_i$ defined by
$\rho_i(\gamma_k) = \gamma_k^i$. Now $\rho_i$ is an epimorphism, since $\pi_1(M_i, p_i)$  
is generated by loops of length at most $2\diam (M_i)\le 2$, and the image of $\rho_i$ 
contains all such elements by construction.
From (\ref{sess}), it follows that
$\Gamma$ contains a finite-index free abelian subgroup $\Gamma'$ of rank $m$.  Call
the index $c$. Then $\rho_i(\Gamma')$ is an abelian subgroup of $\Gamma_i$
(of index at most $c$) generated by at most $n$ elements. 
This contradicts our assumptions and proves
Theorem~\ref{thm-lower-ricci-noncol}. \qed
\small
\bibliographystyle{alpha}
%\bibliography{master-april-2016}

\end{document}